\documentclass[11pt,a4paper,reqno]{amsart} 

\usepackage{amsfonts}
\usepackage{amssymb}
\usepackage{amsmath}
\usepackage{amsthm}
\usepackage{cite}
\usepackage{fancyhdr}
\usepackage[colorlinks,linkcolor=blue]{hyperref}
\usepackage{verbatim}

\theoremstyle{plain}
\newtheorem{Thm}{Theorem}[section]
\newtheorem{Pro}{Proposition}[section]
\newtheorem{Def}{Definition}[section]
\newtheorem{Rk}{Remark}[section]
\newtheorem{Lem}{Lemma}[section]
\newtheorem{Ex}{Example}

\theoremstyle{definition}

\newcounter{marnote}

\numberwithin{equation}{section}
\begin{document}
\title[Hessian quotient equations]{Hessian quotient equations on exterior domains}

\author[H.G. Li]{Haigang Li}
\author[X.L. Li]{Xiaoliang Li}
\author[S.Y. Zhao]{Shuyang Zhao}

\address[H.G. Li]{School of Mathematical Sciences\\
Beijing Normal University \\
	Beijing, 100875\\
	P.R. China}
\email{{\tt  hgli@bnu.edu.cn}}

\address[X.L. Li]{School of Mathematical Sciences\\
Beijing Normal University \\
	Beijing, 100875\\
	P.R. China}
\email{{\tt rubiklixiaoliang@163.com} \ (Corresponding author)}

\address[S.Y. Zhao]{School of mathematical Sciences\\
Peking University\\
Beijing, 100871\\
P.R. China}
\email{{\tt 1701210096@pku.edu.cn}}


\begin{abstract}
It is well-known that a celebrated J\"{o}rgens-Calabi-Pogorelov theorem for Monge-Amp\`ere equations states that any classical (viscosity) convex solution of $\det(D^2u)=1$ in $\mathbb{R}^n$ must be a quadratic polynomial. Therefore, it is an interesting topic to study the existence and uniqueness theorem of such fully nonlinear partial differential equations' Dirichlet problems on exterior domains with suitable asymptotic conditions at infinity. As a continuation of the works of Caffarelli-Li for Monge-Amp\`ere equation and of Bao-Li-Li for $k$-Hessian equations, this paper is devoted to the solvability of the exterior Dirichlet problem of Hessian quotient equations $\sigma_k(\lambda(D^2u))/\sigma_l(\lambda(D^2u))=1$ for any $1\leq l<k\leq n$ in all dimensions $n\geq 2$. By introducing the concept of generalized symmetric subsolutions and then using the Perron's method, we establish the existence theorem for viscosity solutions, with prescribed asymptotic behavior which is close to some quadratic polynomial at infinity. 
\end{abstract}
\maketitle
\section{Introduction}\label{sec:intro}
In this paper, we study the solvability of the exterior Dirichlet problem for Hessian quotient equations
\begin{equation}\label{eq:pro}
\left\{
\begin{array}{ll}
S_{k,l}(D^2u):=\frac{\sigma_k\left(\lambda\left(D^2u\right)\right)}{\sigma_l\left(\lambda\left(D^2u\right)\right)}=1,\quad& in \quad\mathbb{R}^n\setminus\overline{D},\\
u=\varphi, & on\quad\partial D,
\end{array}
\right.
\end{equation}
where $D$ is a bounded open set in $\mathbb{R}^n$, $n\geq2$, $1\leq l<k\leq n$, and $\lambda\left(D^2u\right)$ denotes the eigenvalues $\lambda=(\lambda_1,\cdots,\lambda_n)$ of the Hessian matrix of $u$. Here 
$$\sigma_k(\lambda):=\sum_{1\leq i_1<\cdots<i_k\leq n}\lambda_{i_1}\cdots\lambda_{i_k}$$ is the $k$-th elementary symmetric function of $n$ variations, $k=1,\cdots,n$. This paper is a continuation to the work by Bao, Li and Li \cite{Bao-Li-Li-2014}, where the $k$-Hessian equations was considered.

 As a class of  nonlinear elliptic second-order equations of the form  $F(D^2u)=1$, \eqref{eq:pro} includes several typical cases. In particular, if $l=0$, then it is the classical Poisson equation $\Delta u=1$ when $k=1$, here we set $\sigma_0(\lambda)\equiv1$; when $2\leq k\leq n-1$, we have $k$-Hessian equations $\sigma_k\left(\lambda\left(D^2u\right)\right)=1$; and when $k=n$, it is the well-known Monge-Amp\`ere equation $\det(D^2u)=1$. If $l=1$ and $k=n=3$, then it happens to be the special Lagrangian equation $\det(D^2u)=\Delta u$ in $\mathbb{R}^3$ which originates from the study on the Lagrangian geometry \cite{Harvey-Lawson-1982}. If $l=n-1$ and $k=n$, then it is equivalent to the inverse harmonic Hessian equation $\frac{1}{\lambda_1(D^2u)}+\cdots+\frac{1}{\lambda_n(D^2u)}=1$. 
 
The traditional (interior) Dirichlet problems for these equations have been extensively studied by many mathematicians since 1950s; one can see \cite{Aleksandrov-1958, Nirenberg-1953,Calabi-1958,Cheng-Yau-1977,Caffarelli-1984,Caffarelli-1985,Trudinger-1995, Trudinger-1997,Trudinger-2008, Urbas-1988,Urbas-1990,Wang-2009} and the references therein. Focusing on the topics of exterior Dirichlet problem, there has been many significant progresses in recent years. We would like to mention the work of Caffarelli and Li \cite{Caffarelli-Li-2003} on Monge-Amp\`ere equation, which is an extension of celebrated J\"{o}rgens-Calabi-Pogorelov theorem (\cite{Jorgens-1954, Calabi-1958, Pogorelov-1972, Cheng-Yau-1986, Jost-2001,Caffarelli-1995}) stating that any classical (viscosity) convex solution of $\det(D^2u)=1$ in $\mathbb{R}^n$ must be a quadratic polynomial. They showed that any convex viscosity solution of $\det(D^2u)=1$ in exterior domains of $\mathbb{R}^n, n\geq 3$ satisfies
\begin{equation}
\label{eq:C-Li}
\limsup_{|x|\to\infty}\left(|x|^{n-2}\left|u(x)-\left(\frac12x^TAx+b\cdot x+c\right)\right|\right)<\infty,
\end{equation}
for some symmetric positive definite matrix $A$ with $\det A=1$. In terms of prescribed asymptotic behavior \eqref{eq:C-Li}, they also established the existence and uniqueness theorem to the exterior Dirichlet problem. In $\mathbb{R}^2$, the similar problem was studied in \cite{Philippe-1992,Ferrer-1999, Ferrer-2000,Bao-Li-2012}. Very recently, Li and Lu \cite{Li-Lu-2018} completed the characterization of the existence and nonexistence of solutions with \eqref{eq:C-Li}. 

For $k$-Hessian equations and Hessian quotient equations \eqref{eq:pro}, unlike the Monge-Amp\`ere case, they are not invariant under affine transformations. When $A$ in the prescribed asymptotic condition is a diagonal matrix, the corresponding existence theory in an exterior domain was investigated in \cite{Dai-Bao-2011, Bao-Li-2013}. As for the general positive definite matrix $A$, using a famliy of generlized symmetric subsolutions, Bao, Li and Li \cite{Bao-Li-Li-2014} successfully generalized the existence results of Monge-Amp\`ere equation \cite{Caffarelli-Li-2003} to the exterior Dirichlet problem of $k$-Hessian equations. Lately, Li, Li and Yuan \cite{Li-Li-Yuan-2019} studied the asymptotics \eqref{eq:C-Li} for convex solutions of $2$-Hessian equations and of inverse harmonic equations in exterior domains. We would like to point out that the Liouville type result for global solutions to \eqref{eq:pro} with $k=n$ and $l<n$ has been discussed by Bao et al. \cite{Bao-2003}.

As a continuation of \cite{Bao-Li-Li-2014}, in this paper, we shall establish the existence theorem of the exterior Dirichlet problem for Hessian quotient equations \eqref{eq:pro}, with some prescribed asymptotic behavior at infinity. We follow the notations and definitions in \cite{Caffarelli-1985,Trudinger-1995}.

 First, we say that a function $u\in C^2(\mathbb{R}^n\setminus\overline{D})$ is admissible (or $k$-convex) if $\lambda(D^2u)\in\overline{\Gamma}_k$ in $\mathbb{R}^n\setminus\overline{D}$, where $\Gamma_k$ is the connected component of $\{\lambda\in\mathbb{R}^n|\sigma_k(\lambda)>0\}$ containing 
$$\Gamma^{+}=\{\lambda\in\mathbb{R}^n~|~\lambda_i>0, i=1,\cdots,n\}.$$ Moreover,
$$\Gamma_k=\{\lambda\in\mathbb{R}^n~|~\sigma_j(\lambda)>0, 1\leq j\leq k\}.$$
Then we use $\mathrm{USC}(\Omega)$ and $\mathrm{LSC}(\Omega)$ to respectively denote the set of upper and lower semicontinuous real valued functions on $\Omega\subset\mathbb{R}^n$. The definition of viscosity solution of Hessian quotient equations is given by following.

\begin{Def}
A function $u\in \mathrm{USC}(\mathbb{R}^n\setminus\overline{D})$ is said to be a viscosity subsolution (supersolution) of equation \eqref{eq:pro} (or say that $u$ satisfies $S_{k,l}(D^2u)\geq(\leq)1$ in the viscosity sense), if for any $k$-convex function $\psi\in C^2(\mathbb{R}^n\setminus\overline{D})$ and point $\bar{x}\in \mathbb{R}^n\setminus\overline{D}$  satisfying $$\psi(\bar{x})=u(\bar{x})\quad\mathrm{and}\quad \psi\geq(\leq)u,$$ we have $$S_{k,l}(D^2\psi(\bar{x}))\geq(\leq)1.$$
A function $u\in C^0(\mathbb{R}^n\setminus\overline{D})$ is said to be a viscosity solution of \eqref{eq:pro} if it is both a viscosity subsolution and supersolution of \eqref{eq:pro}.
\end{Def}

\begin{Def}
Let $\varphi\in C^0(\partial D)$. A function $u \in \mathrm{USC}(\mathbb{R}^n\setminus D)$ $(u \in \mathrm{LSC}(\mathbb{R}^n\setminus D))$ is said to be a viscosity subsolution (supersolution) of problem \eqref{eq:pro}, if $u$ is a viscosity subsolution (supersolution) satisfying equation \eqref{eq:pro} in $\mathbb{R}^n\setminus\overline{D}$ and $u\leq (\geq)\varphi$ on $\partial D$. A function $u\in C^0(\mathbb{R}^n\setminus D)$ is said to be a viscosity solution of \eqref{eq:pro} if it is both a subsolution and a supersolution.
\end{Def}

We introduce the concept of generalized symmetric solutions to \eqref{eq:pro} in the following sense.
\begin{Def}
For a diagonal matrix $A=\mathrm{diag}(a_1,a_2,\cdots,a_n)$, we call $u$ a G-Sym (generalized symmetric) function with respect to $A$ if it is a function of $s=\frac12x^TAx=\frac12\sum_{i=1}^na_ix_i^2$, that is, $$u(x)=u(s):=u(\frac12x^TAx).$$ If $u$ is a solution of Hessian quotient equation \eqref{eq:pro} and is also a G-Sym function with respect to $A$, we say that $u$ is a G-Sym solution of \eqref{eq:pro}.
\end{Def}
 
Let 
\begin{align*}\mathcal{A}_{k,l}:=\Big\{A~|~A ~\mbox{is a real $n\times n$ symmetric }&\mbox{positive definite matrix,}~ \\&\mbox{with }~\sigma_k(\lambda(A))=\sigma_l(\lambda(A))~\Big\}.
\end{align*} Note that $c^*(k,l)I\in\mathcal{A}_{k,l}$, where 
\begin{equation}\label{c*}
c^*(k,l):=\left(C_n^l/C_n^k\right)^{\frac{1}{k-l}},
\end{equation} and
$$C_n^k=\frac{n!}{(n-k)!k!},\quad\,C_n^l=\frac{n!}{(n-l)!l!}$$ are two binomial coefficients. In order to avoid the abuse of symbol $\lambda$, we use $a:=(a_1,\cdots,a_n)$ to denote the eigenvalues $\lambda(A)$ in what follows. Clearly, for a diagonal matrix $A=\mathrm{diag}(a_1,\cdots,a_n)\in\mathcal{A}_{k,l}$ and any real constant $\mu$, it is obvious that
\begin{equation}\label{eq:linear-s}
w(s)=s+\mu, \quad\mbox{where}~ s=\frac12x^TAx=\frac12\sum_{i=1}^{n}a_{i}x_{i}^{2}
\end{equation}
is a G-Sym solution of \eqref{eq:pro} and satisfies $w''(s)\equiv0$. Except \eqref{eq:linear-s}, for looking for more G-Sym solutions of \eqref{eq:pro}, we have the following rigidity result.
\begin{Pro}\label{pro:general-solution}
Let $A=\mathrm{diag}(a_1,a_2,\cdots,a_n)\in\mathcal{A}_{k,l}$, $0\leq l<k\leq n$, and $0<\alpha<\beta<\infty$. Then there exists an $w\in C^2(\alpha,\beta)$ with $w''\not\equiv 0$ in $(\alpha,\beta)$, such that $w(x)=w(\frac12\sum_{i=1}^na_ix_i^2)$ is a G-Sym solution of equation \eqref{eq:pro} in $\{x\in\mathbb{R}^n|\alpha<\frac12\sum_{i=1}^na_ix_i^2<\beta\}$, if and only if \begin{equation}\label{rigid}
l=0~\mbox{and}~ k=n,\quad\mbox{or}\quad a_1=a_2=\cdots=a_n=c^*(k,l),
\end{equation}
where $c^*(k,l)=\left(C_n^l/C_n^k\right)^{\frac{1}{k-l}}$ is definded by \eqref{c*}.
\end{Pro}

This means that for $A\in\mathcal{A}_{k,l}$, \eqref{eq:pro} in general has no G-Sym solution unless the two cases in \eqref{rigid} hold. In order to use Perron's method to establish the existence theorem for problem \eqref{eq:pro}, it suffices to obtain enough subsolutions
with appropriate properties. We construct such subsolutions which are G-Sym functions with respect to $A$. This is the main new ingredient.

\begin{Rk}We remark that Proposition \ref{pro:general-solution} is an extension of Proposition 1.4 in \cite{Bao-Li-Li-2014} for $k$-Hessian equations. In particular, the Monge-Amp\`ere equation $\det(D^2u)=1$ has following radially symmetric solutions (see \cite{Caffarelli-Li-2003, Dai-Bao-2011}):
$$\omega_n(\frac12|x|^2)=\int_1^{\frac{|x|^2}{2}}\left(1+\alpha t^{-\frac{n}{2}}\right)^{\frac{1}{n}}\,dt,\quad \alpha>0,$$
which of course is a family of G-Sym solutions for each $A\in\mathcal{A}_{n,0}$ due to the invariance of affine transformations. It is because of this, that the radially symmetric solutions play an important role in the solvability of the exterior Dirichlet problems studied by Caffarelli-Li \cite{Caffarelli-Li-2003}.  However, for $k$-Hessian equations and Hessian quotient equations, we are only allowed to assume that A is diagonal, but we cannot further assume that $A = c^{*}I$. This is the reason why we introduce the G-Sym functions.
\end{Rk}

In order to state our results precisely, we introduce some notations. For any fixed $t$-tuple $\{i_1,\cdots,i_t\}$, $1\leq t\leq n-k$,  we set
$$\sigma_{k;i_1\cdots i_t}(a)=\sigma_k(a)|_{a_{i_1}= a_{i_2}=\cdots=a_{i_t}=0}.$$
We further define
\begin{equation}\label{eq:H-h}
H_k=H_k(\lambda(A)):=\max_{1\leq i\leq n}\frac{\sigma_{k-1;i}(a)a_i}{\sigma_k(a)},\  h_l=h_l(\lambda(A)):=\min_{1\leq i\leq n}\frac{\sigma_{l-1;i}(a)a_i}{\sigma_l(a)}.
\end{equation}
Set $h_0\equiv 0$ for completeness. By Perron's method, our first principal result is as follows.

\begin{Thm}\label{thm:main-1}
Let $D$ be a smooth, bounded, strictly convex open subset of $\mathbb{R}^n$, $n\geq 3$ and let $\varphi\in C^2(\partial D)$. For any $A\in\mathcal{A}_{k,l}$ with $1\leq l<k\leq n$ and $b\in\mathbb{R}^n$, if 
\begin{itemize}
\item [(i)] $k-l\geq2$, or 
\item [(ii)] $k-l=1$ and $H_k-h_l<\frac12$,
\end{itemize}
then there exists some constant $c_*$ depending only on $n,b,A,D$ and $\|\varphi\|_{C^2(\partial D)}$, such that for every $c>c_*$ there exists a unique viscosity solution $u\in C^0(\mathbb{R}^n\setminus D)$ of \eqref{eq:pro} and 
\begin{equation}\label{eq:asym-1}
\limsup_{|x|\to\infty}\left(|x|^{\frac{k-l}{H_k-h_l}-2}\Big{|}u(x)-(\frac12x^TAx+b\cdot x+c)\Big{|}\right)<\infty.
\end{equation}  
Especially, when $k-l\geq2$, \eqref{eq:asym-1} can be written as 
\begin{equation}\label{eq:asym-2}
\limsup_{|x|\to\infty}\left(|x|^{\theta(n-2)}\Big{|}u(x)-(\frac12x^TAx+b\cdot x+c)\Big{|}\right)<\infty.
\end{equation}  
where $\theta\in (\frac{k-l-2}{n-2},1]$ is a constant depending only on $n,k,l$ and $A$.
\end{Thm}

\begin{Rk}
When $l=0$, the asymptotics \eqref{eq:asym-2} is consistent in the well-known results. For instance, when $k=n$, for the Monge-Amp\`ere equation, see Caffarelli-Li \cite{Caffarelli-Li-2003}; when $2\leq k\leq n-1$, the Hessian equations, see Bao-Li-Li \cite{Bao-Li-Li-2014}. On the other hand, when $A=c^*(k,l)I$, then $H_k-h_l=\frac{k-l}{n}$, see \cite{Bao-Li-2013,Dai-Bao-2011}. We would like to point out that Theorem \ref{thm:main-1} is the main result of \cite{Zhao-2017}. Another independent proof is also given in \cite{Li-Li-2017}.
\end{Rk}

When $k=n$ and $l=n-1$, for the inverse harmonic equations in exterior domains, by using a different technique of constructing subsolutions, we have the following existence and uniqueness theorem to the exterior problem 
\begin{equation}\label{eq:pro-i}
\left\{
\begin{array}{ll}
S_{n,n-1}(D^2u)=1,& in \quad\mathbb{R}^n\setminus\overline{D},\\
u=\varphi, & on\quad\partial D.
\end{array}
\right.
\end{equation}

\begin{Thm}\label{thm:P-L}
Let $D$ be a smooth, bounded, strictly convex open subset in $\mathbb{R}^n$ and let $\varphi\in C^2(\partial D)$. Then for any $A\in\mathcal{A}_{n,n-1}$, $b\in\mathbb{R}^n$, 

(i) when $n\geq 3$, for each $\gamma<0$, there exists some constant $c_*$ depending only on $n,\gamma,A,b,D$ and $\|\varphi\|_{C^2(\partial D)}$, such that for every $c>c_*$,  there exists a unique viscosity solution $u\in C^0(\mathbb{R}^n\setminus D)$ of \eqref{eq:pro-i} fulfilling 
\begin{equation*}
\limsup_{|x|\to\infty}\left(|x|^{-\gamma}\left|u(x)-(\frac12x^TAx+b\cdot x+c)\right|\right)<\infty;
\end{equation*}

(ii) when $n=2$, there exists some constant $\alpha_*$ depending only on $A,b,D$ and $\|\varphi\|_{C^2(\partial D)}$, such that for every $\alpha>\alpha_*$, there exists a unique local convex solution $u\in C^\infty(\mathbb{R}^2\setminus \overline{D})\cap C^0(\mathbb{R}^2\setminus D)$ of \eqref{eq:pro-i} fulfilling 
\begin{equation*}
O(|x|^{-2})\leq u(x)-V(x)\leq M(\alpha)+O(|x|^{-2}),\quad\text{ as }|x|\to\infty,
\end{equation*}  
where 
\begin{gather*}
V(x)=\frac{1}{2}x^TAx+b\cdot x+\alpha\ln\sqrt{x^T(A-I)x}+c(\alpha),
\end{gather*}
 and $M(\alpha), c(\alpha)$ are functions of $\alpha$.
\end{Thm}

\begin{Rk}
Here the restriction $H_{n}-h_{n-1}<\frac12$ is not required. The reason why we do not establish the existence for other cases $k-l=1$ and $H_k-h_l\geq\frac12$ is purely technical, and more discussions can be refered to the Appendix.
\end{Rk}

The remainder of this paper is organized as follows. As explained in \cite{Bao-Li-Li-2014}, in order to prove Theorems \ref{thm:main-1} and \ref{thm:P-L}, we only need to show that the assertion holds when $A$ is a diagonal matrix and $b=0$ by making use of an orthogonal transformation and by subtracting a linear function from $u$; but we cannot further assume that $A=c^*(k,l)I$ because the Hessian quotient equations are not invariant under affine transformations, like the Monge-Amp\`ere equation \cite{Caffarelli-Li-2003}. 
In the next section, we first demonstrate that the Hessian quotient equations \eqref{eq:pro} with $1\leq l<k\leq n$ do not have G-Sym solutions with respect to $A$ for every $A\in\mathcal{A}_{k,l}$ unless $A=c^*(k,l)I$, see Proposition \ref{pro:general-solution}. Based on this rigidity property and to apply the Perron's method,  the heart of our proof is to construct appropriate subsolutions which are G-Sym functions with respect to $A\in\mathcal{A}_{k,l}$ and verify certain asymptotic behaviors at infinity.  Proposition \ref{pro:omega-alpha} provides us a family of G-Sym $k$-convex subsolutions of \eqref{eq:pro} with respect to $A$ for every $A\in\mathcal{A}_{k,l}$, no matter whether $H_k-h_l<\frac{k-l}{2}$ or not. Then by using Perron's method, we prove Theorem \ref{thm:main-1} and Theorem \ref{thm:P-L} in Section \ref{sec:3} and Section \ref{sec:4}, respectively. At last, in the Appendix,  
we give several examples with $H_k-h_l=\frac12$ in the phase space of eigenvalues $\lambda(A)$ and explain more why we are not able to apply the Perron's method to build the existence result of viscosity solutions for problem \eqref{eq:pro} when $k-l=1$ and $H_k-h_l\geq \frac12$ with $k<n$. To solve it, new technique is needed.

\section{G-Sym solutions and G-Sym subsolutions}\label{sec:2} 
This section is mainly concerned with the G-Sym solutions and subsolutions of equation \eqref{eq:pro} with respect to $A\in\mathcal{A}_{k,l}$, $0\leq l<k\leq n$.  First, we  show that for a diagonal matrix $A\in\mathcal{A}_{k,l}$, in general, there is no G-Sym solution, analogous to that in \cite{Bao-Li-Li-2014} using similar techniques. Then we utilize the $H_k$ and $h_l$ defined in \eqref{eq:H-h} to construct a family of G-Sym $k$-convex subsolutions in $\mathbb{R}^n\setminus\{0\}$. This will be important to prove Theorem \ref{thm:main-1} in the next section.

We start with recalling some elementary properties of $\sigma_k(a)$. Suppose $1\leq l\leq k\leq n$ and $a=(a_1,\cdots, a_n)$ with $a_i>0$, $i=1,\cdots,n$. Then 
\begin{gather}
\sigma_k(a)=\sigma_{k;i}(a)+a_i\sigma_{k-1;i}(a), \quad\forall~i,\label{eq:sigma-k}\\
\sum_{i=1}^na_i\sigma_{k-1;i}(a)=k\sigma_{k}(a),\label{eq:k-sigma}\\
\sigma_l(a)\sigma_{k-1;i}(a)\geq \sigma_{l-1;i}(a)\sigma_k(a),\quad\forall~i. \label{eq:sigma-l-k}
\end{gather}

\subsection{G-Sym solutions and the proof of Proposition \ref{pro:general-solution}}

We here make use of the idea of Bao-Li-Li \cite{Bao-Li-Li-2014} to prove Proposition \ref{pro:general-solution}. The following lemma is also needed.

\begin{Lem}[Lemma 1.3 in \cite{Bao-Li-Li-2014}]\label{lem:D^2-omega}
For any $A=\mathrm{diag}(a_1,\cdots,a_n)$, if $w\in C^2(\mathbb{R}^n)$ is a G-Sym function with respect to $A$, then, with $a:=(a_1,\cdots,a_n)$,
\begin{equation}\label{eq:D^2-omega}
\sigma_k(\lambda(D^2w))=\sigma_k(a)(w')^k+w''(w')^{k-1}\sum_{i=1}^n\sigma_{k-1;i}(a)(a_ix_i)^2.
\end{equation}
\end{Lem}  

\begin{proof}[Proof of Proposition \ref{pro:general-solution}]
When $l=0$ and $1\leq k\leq n$, for $k$-Hessian equations case, Proposition \ref{pro:general-solution} has been proved in \cite{Bao-Li-Li-2014}.
We here consider the Hessian quotient cases with $1\leq l<k\leq n$.

Given a fixed $s\in(\alpha,\beta)$ such that $w'(s)\neq 0$ and $w''(s)\neq 0$, take some integer $1\leq i\leq n$ and let $$x=(0,\cdots,0,\sqrt{2s/a_i},0,\cdots,0).$$ 

 
 From \eqref{eq:D^2-omega}, we see that the G-Sym solution $w(s)$ of \eqref{eq:pro} satisfies the following ordinary equation
\begin{equation*}
\sigma_k(a)(w')^k+2sw''(w')^{k-1}\sigma_{k-1;i}(a)a_i=\sigma_l(a)(w')^l+2sw''(w')^{l-1}\sigma_{l-1;i}(a)a_i.
\end{equation*}
We rewrite it as 
$$\frac{(w')^l-(w')^k}{2sw''}=(w')^{k-1}\frac{\sigma_{k-1;i}(a)a_i}{\sigma_k(a)}-(w')^{l-1}\frac{\sigma_{l-1;i}(a)a_i}{\sigma_l(a)}.$$
Notice that the left hand side of the above equality depends only on $s$ and is independent of $i$, so we have, for any $i\neq j$,
\begin{align*}
&(w')^{k-1}\frac{\sigma_{k-1;i}(a)a_i}{\sigma_k(a)}-(w')^{l-1}\frac{\sigma_{l-1;i}(a)a_i}{\sigma_l(a)}\\
={}&(w')^{k-1}\frac{\sigma_{k-1;j}(a)a_j}{\sigma_k(a)}-(w')^{l-1}\frac{\sigma_{l-1;j}(a)a_j}{\sigma_l(a)}.
\end{align*}
That is,
$$
 (w')^{k-1}\frac{\sigma_{k-1;i}(a)a_i-\sigma_{k-1;j}(a)a_j}{\sigma_k(a)}=(w')^{l-1}\frac{\sigma_{l-1;i}(a)a_i-\sigma_{l-1;j}(a)a_j}{\sigma_l(a)}.
$$
By using \eqref{eq:sigma-k} and $\sigma_k(a)=\sigma_l(a)$, we have
\begin{equation}\label{eq:w-i-j}
(w')^{k-1}[(a_i-a_j)\sigma_{k-1;ij}(a)]=(w')^{l-1}[(a_i-a_j)\sigma_{l-1;ij}(a)].
\end{equation}

If $k=n$, 
then $\sigma_{k-1;ij}(a)=0$ for any $i\neq j$. Due to $\sigma_{l-1;ij}(a)>0$, we immediately get $a_i=a_j$ from \eqref{eq:w-i-j}. Thus, by arbitrariness of $i$ and $j$, we see $a_1=a_2=\cdots=a_n=c^*(n,l)$. 

Now, we consider the case $1\leq l<k\leq n-1$. Assume for contradiction that $a_i\neq a_j$ for some $i\neq j$. From \eqref{eq:w-i-j}, it follows that
$$(w')^{k-1}\sigma_{k-1;ij}(a)=(w')^{l-1}\sigma_{l-1;ij}(a).$$
So that,
\begin{equation}\label{eq:i-j-C-s}
\frac{\sigma_{k-1;ij}(a)}{\sigma_{l-1;ij}(a)}=(w'(s))^{l-k}.
\end{equation}
Since the left side of \eqref{eq:i-j-C-s} is independent of $s$, we deduce that $w'(s)$ is a constant, which leads to $w''(s)\equiv 0$, a contradiction. Consequently, $a_1=a_2=\cdots=a_n=c^*(k,l)$.
\end{proof}

\subsection{G-Sym subsolutions}
By Proposition \ref{pro:general-solution}, it is only possible to find some   G-Sym smooth $k$-convex  subsolutions of equation \eqref{eq:pro} for any $A\in\mathcal{A}_{k,l}$. As before, we let $A=\mathrm{diag}(a_1,\cdots,a_n)\in\mathcal{A}_{k,l}$ and denote $\lambda(A)=(a_1,\cdots, a_n):=a$. We have $a_i>0$ $(i=1,\cdots,n)$ and $\sigma_k(a)=\sigma_l(a)$. Recall \eqref{eq:H-h}, 
$$H_k=\max_{1\leq i\leq n}\frac{\sigma_{k-1;i}(a)a_i}{\sigma_k(a)}, \quad h_l=\min_{1\leq i\leq n}\frac{\sigma_{l-1;i}(a)a_i}{\sigma_l(a)},$$ 
 which indicates from \eqref{eq:sigma-k} and \eqref{eq:k-sigma} that, for all $0\leq l<k\leq n$,
$$\frac{k}{n}\leq H_k\leq1,\quad 0\leq h_l\leq\frac{l}{n},$$
and 
\begin{equation}\label{eq:H-h-k-l}
\frac{k-l}{n}\leq H_k-h_l\leq1.
\end{equation}
In the following, we set 
$$\mathcal{H}_{k,l}:=\frac{k-l}{2(H_k-h_l)},$$ for simplicity.

Now, using Lemma \ref{lem:D^2-omega} and recalling $s=\frac12x^TAx=\frac12\sum_{i=1}^na_ix_i^2$, we study the following ordinary equation
\begin{equation}\label{eq:ordinary-eq}
\left\{
\begin{array}{ll}
(w')^k+2w''(w')^{k-1}H_ks-(w')^l-2w''(w')^{l-1}h_ls=0, & s>0,\\
w'(s)>0, w''(s)<0.
\end{array}
\right.
\end{equation}
By setting $v(s)=w'(s)$, then
$$v^k+v'v^{k-1}H_k2s-v^l-v'v^{l-1}h_l2s=0.$$
Multiplying it by $\mathcal{H}_{k,l}v^{\frac{h_lk-H_kl}{H_k-h_l}}s^{\mathcal{H}_{k,l}-1}$ on both sides, we can rewrite it as
\begin{align*}
&\mathcal{H}_{k,l}v^{2\mathcal{H}_{k,l}H_k}s^{\mathcal{H}_{k,l}-1}+2\mathcal{H}_{k,l}H_kv'v^{2\mathcal{H}_{k,l}H_k-1}s^{\mathcal{H}_{k,l}}\\
=&\mathcal{H}_{k,l}v^{2\mathcal{H}_{k,l}h_l}s^{\mathcal{H}_{k,l}-1}+2\mathcal{H}_{k,l}hv'v^{2\mathcal{H}_{k,l}h_l-1}s^{\mathcal{H}_{k,l}}.
\end{align*}
Then, integrating it on both sides with respect to $s$, we see that
$$
v^{2\mathcal{H}_{k,l}H_k}s^{\mathcal{H}_{k,l}}=v^{2\mathcal{H}_{k,l}h_l}s^{\mathcal{H}_{k,l}}+\alpha,
$$
that is,
\begin{equation}\label{eq:v-k-l-H-h-s}
(v^{k-l}-1)v^{2\mathcal{H}_{k,l}h_l}=\alpha s^{-\mathcal{H}_{k,l}},
\end{equation}
where $\alpha$ is an arbitrary constant. Given $\alpha>0$, after differentiating \eqref{eq:v-k-l-H-h-s}, it is easy to find that there exists a solution $v_\alpha(s)>1$, such that 
\begin{equation}\label{eq:d-v-alpha-s}
\frac{v'_\alpha}{v_\alpha}\left(H_kv_\alpha^{2\mathcal{H}_{k,l}H_k}-h_lv_\alpha^{2\mathcal{H}_{k,l}h_l}\right)=-\frac{\alpha}{2s^{1+\mathcal{H}_{k,l}}}.
\end{equation}
This indicates that $v'_\alpha(s)<0$. Therefore, for any $\alpha>0$, the ordinary equation \eqref{eq:ordinary-eq} has a family of solutions  
\begin{align}
\omega_\alpha(s)&=\beta+\int_{\bar{s}}^sv_\alpha(t)\,dt\notag\\
&=\beta+s-\bar{s}+\int_{\bar{s}}^s\left(v_\alpha(t)-1\right)\,dt, \label{eq:ordinary-solution}
\end{align}
where $\beta\in\mathbb{R}$ and $\bar{s}>0$.

Next, let us characterize the asymptotic behavior of the solutions of \eqref{eq:ordinary-eq}. For $\alpha>0$, letting $s\to\infty$, we deduce by \eqref{eq:v-k-l-H-h-s} that
$$\lim_{s\to\infty}v_\alpha(s)=1,$$ and 
\begin{equation}\label{eq:O-v-s}
v_\alpha(s)-1=O(s^{-\mathcal{H}_{k,l}}), \quad\mathrm{as}\  s\to\infty.
\end{equation}
We divide below into four cases.

\textbf{Case 1.} If $H_k-h_l<\frac{k-l}{2}$, it follows from \eqref{eq:ordinary-solution} and \eqref{eq:O-v-s} that 
\begin{equation}\label{eq:omega-s-1}
\omega_\alpha(s)=s+\mu_1(\alpha)+O(s^{1-\mathcal{H}_{k,l}}),\quad s\to\infty,
\end{equation}
where $$\mu_1(\alpha)=\beta-\bar{s}+\int_{\bar{s}}^\infty(v_\alpha(t)-1)\,dt<\infty.$$ Moreover, in view of \eqref{eq:v-k-l-H-h-s},
\begin{equation*}
\frac{\partial v_\alpha(s)}{\partial\alpha}=\frac{s^{-\mathcal{H}_{k,l}}}{2\mathcal{H}_{k,l}\left(H_kv_\alpha^{2\mathcal{H}_{k,l}H_k-1}-h_lv_\alpha^{2\mathcal{H}_{k,l}h_l-1}\right)}>0,
\end{equation*}
and we can see that $\mu_1(\alpha)$ increases with respect to $\alpha$ and $\lim_{\alpha\to\infty}\mu_1(\alpha)=\infty$. In particular, when $n\geq 3$ and $k-l\geq2$, we have by \eqref{eq:H-h-k-l} that $H_k-h_l<\frac{k-l}{2}$ holds for any $A\in\mathcal{A}_{k,l}$ and 
\begin{equation*}
\frac{k-l}{2}-1<\mathcal{H}_{k,l}-1\leq\frac{n-2}{2}.
\end{equation*}
Thus, in this case, \eqref{eq:omega-s-1} can be written as 
\begin{equation}\label{eq:omega-s-2}
\omega_\alpha(s)=s+\mu_1(\alpha)+O(s^{\frac{\theta(2-n)}{2}}),\quad\theta\in(\frac{k-l-2}{n-2},1].
\end{equation}

\textbf{Case 2.} If $n=k=2$ and $l=0$, then using \eqref{eq:H-h-k-l} we have $H_2-h_0=1$. From \eqref{eq:v-k-l-H-h-s} we immediately obtain $v_\alpha=\sqrt{1+\alpha/s}$ and \eqref{eq:ordinary-solution} becomes 
\begin{align}
\omega_\alpha(s)&=\int_{\bar{s}}^s\sqrt{1+\alpha/s}\,ds+\beta\notag\\
&=\left[\sqrt{t}\sqrt{\alpha+t}+\alpha\ln\left(\sqrt{t}+\sqrt{\alpha+t}\right)\right]\Big{|}_{\bar{s}}^s+\beta\notag\\
&=s+\frac{\alpha}{2}\ln s+\mu_2(\alpha)+O(s^{-1}),\label{eq:omega-ln-2}
\end{align}
where $$\mu_2(\alpha)=\beta-\bar{s}-\frac{\alpha}{2}\ln\bar{s}+\int_{\bar{s}}^\infty\left(v_\alpha-1-\frac{\alpha}{2s}\right)\,dt<\infty.$$

\textbf{Case 3.} If $k-l=1$ and $H_k-h_l=\frac{k-l}{2}=\frac12$. By \eqref{eq:v-k-l-H-h-s} we obtain $(v_\alpha-1)v_\alpha^{2h_l}=\alpha/s$ and $v_\alpha-1=O(s^{-1})$. Since
$$\lim_{s\to\infty}\frac{v_\alpha-1-\alpha/s}{s^{-2}}=-2h_l\alpha^2<\infty.$$
which implies that $v_\alpha-1-\alpha/s=O(s^{-2}), s\to\infty$. We thereby derive from \eqref{eq:ordinary-solution} that
\begin{align}
\omega_\alpha(s) &=\beta+s-\bar{s}+\int_{\bar{s}}^s \alpha/s\,dt+\int_{\bar{s}}^s(v_\alpha-1-\alpha/s)\,dt\notag\\
&=s+\alpha\ln s+\mu_3(\alpha)+O(s^{-1}),\label{eq:omega-ln}
\end{align}
where $$\mu_3(\alpha)=\beta-\bar{s}-\alpha\ln\bar{s}+\int_{\bar{s}}^\infty(v_\alpha-1-\alpha/s)\,dt<\infty.$$

\textbf{Case 4.} If $k-l=1$ and $H_k-h_l>\frac{k-l}{2}=\frac12$, then \eqref{eq:v-k-l-H-h-s} reduces to $(v-1)v^{\frac{h_l}{H_k-h_l}}=\alpha s^{\frac{-1}{2(H_k-h_l)}}$. We easily verify that 
 $$\lim_{s\to\infty}\frac{v_\alpha-1-\alpha s^{\frac{-1}{2(H_k-h_l)}}}{1/s}=0,$$
 and $$\lim_{s\to\infty}\frac{v_\alpha-1-\alpha s^{\frac{-1}{2(H_k-h_l)}}}{1/s^2}=\infty.$$
 Hence, $v_\alpha-1-\alpha s^{\frac{-1}{2(H_k-h_l)}}=O(s^{\theta-1})$, $\theta\in (-1,0)$. Using \eqref{eq:ordinary-solution} gives
 \begin{align}
\omega_\alpha(s)&=\beta+s-\bar{s}+\int_{\bar{s}}^s \alpha s^{\frac{-1}{2(H_k-h_l)}}\,dt+\int_{\bar{s}}^s(v_\alpha-1-\alpha s^{\frac{-1}{2(H_k-h_l)}})\,dt\notag\\
&=s+\frac{\alpha}{1-\frac{1}{2(H_k-h_l)}}s^{1-\frac{1}{2(H_k-h_l)}}+\mu_4(\alpha)+O(s^\theta),\label{eq:omega-power}
\end{align}
where $$\mu_4(\alpha)=\beta-\bar{s}-\frac{\alpha}{1-\frac{1}{2(H_k-h_l)}}\bar{s}^{1-\frac{1}{2(H_k-h_l)}}+\int_{\bar{s}}^\infty(v_\alpha-1-\alpha/s)\,dt<\infty.$$

Now, we are in a position to state explicitly that a family of functions $\omega_\alpha(s)$ given by \eqref{eq:ordinary-solution} are G-Sym subsolutions of equation \eqref{eq:pro} in $\mathbb{R}^n\setminus\{0\}$. Precisely, we have
\begin{Pro}\label{pro:omega-alpha}
For $n\geq2$, $0\leq l<k\leq n$ and $\alpha>0$, let $A\in\mathcal{A}_{k,l}$ be diagonal and $\omega_\alpha(x)=\omega_\alpha(\frac12x^TAx)$ be given by \eqref{eq:ordinary-solution}. Then $\omega_\alpha$ is a smooth $k$-convex subsolution of equation \eqref{eq:pro} in $\mathbb{R}^n\setminus\{0\}$. Furthermore,

(i)  Assume $n\geq 3$. If $k-l=1$ and $H_k-h_l<\frac12$, then 
\begin{equation}\label{eq:omega-alpha-1}
\omega_\alpha(x)=\frac12x^TAx+\mu_1(\alpha)+O(|x|^{2-2\mathcal{H}_{k,l}}), \quad |x|\to\infty;
\end{equation}
if $k-l\geq2$, then
\begin{equation}\label{eq:omega-alpha-2}
\omega_\alpha(x)=\frac12x^TAx+\mu_1(\alpha)+O(|x|^{\frac{\theta(2-n)}{2}}), \quad |x|\to\infty,
\end{equation}
where  $$\theta\in\Big(\frac{k-l-2}{n-2},1\Big].$$

(ii) If $n=k=2$ and $l=0$, then
\begin{equation}\label{eq:omega-alpha-3}
\omega_\alpha(x)=\frac12x^TAx+\frac{\alpha}{2}\ln \left(\frac12x^TAx\right)+\mu_2(\alpha)+O(|x|^{-2}), \quad |x|\to\infty.
\end{equation}

(iii) If $k-l=1$ and $H_k-h_l=\frac12$, then
\begin{equation}\label{eq:omega-alpha-4}
\omega_\alpha(x)=\frac12x^TAx+\alpha\ln \left(\frac12x^TAx\right)+\mu_3(\alpha)+O(|x|^{-2}), \quad |x|\to\infty.
\end{equation}

(iv) If $k-l=1$ and $H_k-h_l>\frac12$, then, as $|x|\to\infty$,
\begin{equation}\label{eq:omega-alpha-5}
\omega_\alpha(x)=\frac12x^TAx+\frac{\alpha}{1-\frac{1}{2(H_k-h_l)}}\left(\frac12x^TAx\right)^{1-\frac{1}{2(H_k-h_l)}}+\mu_4(\alpha)+O(|x|^{2\theta}), 
\end{equation}
where $\theta\in(-1,0)$.
\end{Pro}

\begin{Rk}

(1) We would like to remark that Li and Li  \cite{Li-Li-2017} also obtain the assertion (i), by using variable $r=\sqrt{x^TAx}$. For the case  $l=0$, $2\leq k\leq n$, the G-Sym subsolutions of $k$-Hessian equations \eqref{eq:pro} given by Proposition \ref{pro:omega-alpha} have been clarified by Bao, Li and Li in \cite{Bao-Li-Li-2014}. 

(2) In particular, if we take $A=c^*(k,l)I$,  then 
$$\frac{\sigma_{k-1;i}(a)a_i}{\sigma_k(a)}=\frac{k}{n}, \quad\frac{\sigma_{l-1;i}(a)a_i}{\sigma_l(a)}=\frac{l}{n}, \quad\forall ~i=1,2,\cdots, n,$$
and
$$H_k-h_l=\frac{k-l}{n}.$$
If $n=k=2$ and $l=0$, then $c^*(2,0)=1$, $H_2=1$. By \eqref{eq:omega-ln-2} and \eqref{eq:omega-alpha-3}, we obtain its radial symmetric solution of the Monge-Amp\`ere equation $\det(D^2u)=1$ in dimension two
\begin{equation*}
\omega_\alpha(x)=\frac12\left(|x|\sqrt{2\alpha+|x|^2}\right)+\alpha\ln\left(|x|^2+\sqrt{2\alpha+|x|^2}\right)+C,
\end{equation*}
satisfying $$\omega_\alpha(x)=\frac12|x|^2+\frac{\alpha}{2}\ln(\frac12|x|^2)+\mu_2(\alpha)+O(|x|^{-2}), \quad x\to\infty,$$
where $C$ is a constant depending on $\alpha$, see \cite{Wang-Bao-2013, Bao-Li-2012}. If $n\geq 3$, $2\leq k\leq n$ and $l=0$, then \eqref{eq:ordinary-solution} with $H_k=\frac{k}{n}$ and $h_l=0$ gives the radial symmetric solutions of $k$-Hessian equations \eqref{eq:pro} in $\mathbb{R}^n\setminus B_1$:
$$\omega_\alpha(x)=\int_1^{\frac{c^*(k,0)}{2}|x|^2}\left(1+\alpha t^{-\frac{n}{2}}\right)^{\frac{1}{k}}\,dt.$$
If $n\geq 3$ and $k=1$, then $c^*(1,0)=\frac{1}{n}$ and we obtain the radial symmetric solutions in $\mathbb{R}^n\setminus B_1$ for Poisson equation $\Delta u=1$, 
\begin{align*}
\omega_\alpha(x)&=\int_1^{\frac{1}{2n}|x|^2}\left(1+\alpha t^{-\frac{n}{2}}\right)\,dt\\
&=\frac{1}{2n}|x|^2-\frac{\alpha}{n(n-2)}|x|^{2-n}-1+\frac{2\alpha}{n-2},
\end{align*}
which implies that the radial solution of Poisson equation in $\mathbb{R}^n\setminus\{0\}$ can only be a summation of Laplace's fundamental solution $\frac{\alpha}{n(2-n)}|x|^{2-n}$ and the solution $\frac{1}{2n}|x|^2+c$ in the form \eqref{eq:linear-s}. 

\end{Rk}
 
We conclude this section by giving the proof of Proposition \ref{pro:omega-alpha}.
\begin{proof}[Proof of Proposition \ref{pro:omega-alpha}]
Clearly, \eqref{eq:omega-alpha-1}-\eqref{eq:omega-alpha-5} directly follow from \eqref{eq:omega-s-1}, \eqref{eq:omega-s-2}, \eqref{eq:omega-ln-2}, \eqref{eq:omega-ln} and \eqref{eq:omega-power}. We denote $\lambda(A)=(a_1,\cdots,a_n):=a$ throughout the proof. Since $v_\alpha'(s)=\omega_\alpha''(s)<0$, for any $1\leq m\leq k$,
\begin{align*}
\sigma_m(\lambda(D^2\omega_\alpha))&=\sigma_m(a)v_\alpha^m+v'_\alpha v_\alpha^{m-1}\sum_{i=1}^n\sigma_{m-1;i}(a)(a_ix_i)^2\\
&=\sigma_m(a)v_\alpha^{m-1}\left(v_\alpha+v'_\alpha\sum_{i=1}^n\frac{\sigma_{m-1;i}(a)}{\sigma_m(a)}(a_ix_i)^2\right)\\
&\geq \sigma_m(a)v_\alpha^{m-1}(v_\alpha+v'_\alpha H_m2s).
\end{align*}
From \eqref{eq:sigma-l-k} and \eqref{eq:A-H-h} , we know that $$H_m=\frac{\sigma_{m-1;n}(a)a_n}{\sigma_m(a)}\leq \frac{\sigma_{k-1;n}(a)a_n}{\sigma_k(a)}=H_k\quad \mathrm{for}\quad 1\le m\le k,$$  provided $a_1\le a_2\le\cdots\le a_n$. Hence,
$$\sigma_m(\lambda(D^2\omega_\alpha))\ge \sigma_m(a)v_\alpha^{m-1}(v_\alpha+v'_\alpha H_k2s).$$ 

In view of \eqref{eq:v-k-l-H-h-s} and \eqref{eq:d-v-alpha-s},
\begin{align*}
\frac{v'_\alpha}{v_\alpha}\frac{\alpha}{s^{\mathcal{H}_{k,l}}}H_k&=\frac{v'_\alpha}{v_\alpha}\left(H_kv^{2\mathcal{H}_{k,l}H_k}-H_kv^{2\mathcal{H}_{k,l}h_l}\right)\\
&\geq\frac{v'_\alpha}{v_\alpha}\left(H_kv^{2\mathcal{H}_{k,l}H_k}-h_lv^{2\mathcal{H}_{k,l}h_l}\right) =-\frac{\alpha}{2s^{\mathcal{H}_{k,l}+1}}.
\end{align*}
It shows that $v_\alpha+v'_\alpha H_k2s\geq 0$ and so $$\sigma_m(\lambda(D^2\omega_\alpha))>0, \quad x\in\mathbb{R}^n\setminus\{0\},$$ for any $1\leq m\leq k$.

 On the other hand, by the definition of $H_k,h_l$ and the fact that $\omega''_\alpha(s)<0$,
\begin{align*}
&\sigma_k(\lambda(D^2\omega_\alpha))-\sigma_l(\lambda(D^2\omega_\alpha))\\
={}&\sigma_k(a)(\omega'_\alpha)^k+\omega''_\alpha(\omega'_\alpha)^{k-1}\sum_{i=1}^n\sigma_{k-1;i}(a)(a_ix_i)^2\\
&\quad-\sigma_l(a)(\omega'_\alpha)^l-\omega''_\alpha(\omega'_\alpha)^{l-1}\sum_{i=1}^n\sigma_{l-1;i}(a)(a_ix_i)^2\\
\geq{} & \sigma_k(a)\left((\omega'_\alpha)^k+2\omega''_\alpha(\omega'_\alpha)^{k-1}H_ks-(\omega'_\alpha)^l-2\omega''_\alpha(\omega'_\alpha)^{l-1}h_ls\right)=0.
\end{align*}
Consequently, $\omega_\alpha$ is a smooth $k$-convex subsolution of \eqref{eq:pro} in $\mathbb{R}^n\setminus\{0\}$.
\end{proof}

\section{Proof of Theorem \ref{thm:main-1}}\label{sec:3}
To prove Theorem \ref{thm:main-1}, we need apply an adapted Perron's method and comparison principle for general equation $f(\lambda(D^2(u)))=1$ to Hessian quotient equations \eqref{eq:pro}, see \cite{Li-Bao-2014, Bao-Li-Li-2014} and the references therein. For convenience, we present them as follows.

\begin{Lem}\label{lem:perron-m}
Let $\Omega$ be a domain in $\mathbb{R}^n$. Assume that there exist $\underline{u}, \bar{u}\in C^0(\overline{\Omega})$ respectively to be viscosity subsolution and supersolution of \eqref{eq:pro} such that $\underline{u}\leq \bar{u}$, and $\underline{u}=\varphi$ on $\partial \Omega$. Then
$$u(x):=\sup\{v(x)|\underline{u}\leq v\leq\bar{u}\, \mathrm{in}\, \Omega\,\mathrm{and}\, v\, \mbox{is a subsolution of \eqref{eq:pro}}, \text{with } v=\varphi\,\mathrm{on}\, \partial\Omega\}$$
is the unique viscosity solution of problem \eqref{eq:pro}.
\end{Lem}

\begin{Lem}\label{lem:compar}
Let $\Omega$ be a domain in $\mathbb{R}^n$. If $u\in\mathrm{USC}(\overline{\Omega})$, $v\in\mathrm{LSC}(\overline{\Omega})$ are respectively viscosity subsolution and supersolution of \eqref{eq:pro} in $\Omega$ and $u\leq v$ on $\partial\Omega$, then $u\leq v$ in $\Omega$.
\end{Lem}

 Besides, we also need the following lemma which has proved in \cite{Caffarelli-Li-2003,Bao-Li-Li-2014}.  

\begin{Lem}\label{lem:w-xi}
Let $D$ be a bounded strictly convex domain of $\mathbb{R}^n, n\geq2$, $\partial D\in C^2$, $\varphi\in C^2(\partial D)$ and let $A$ be an invertible and symmetric matrix. There exists some constant $C$, depending only on $n, \|\varphi\|_{C^2(\partial D)}$, the upper bound of $A$, the diameter and the convexity of $D$, and the $C^2$ norm of $\partial D$, such that for every $\xi\in\partial D$, there exists $\bar{x}(\xi)\in\mathbb{R}^n$ satisfying $$|\bar{x}(\xi)|\leq C\quad\text{and}\quad w_\xi<\varphi\quad\text{on}\quad \partial D\setminus\{\xi\},$$ where $$w_\xi(x)=\varphi(\xi)+\frac12\left((x-\bar{x}(\xi))^TA(x-\bar{x}(\xi))-(\xi-\bar{x}(\xi))^TA(\xi-\bar{x}(\xi))\right),\ x\in\mathbb{R}^n.$$
\end{Lem}

 We now start to prove Theorem \ref{thm:main-1}, provided $A=\text{diag}(a_1,a_2,\cdots,a_n)$ and $b=0$. Actually, using Lemmas \ref{lem:perron-m}-\ref{lem:w-xi}, the proof is similar to that of  \cite{Bao-Li-Li-2014} for $k$-Hessian equations. For the reader's convenience, we here present it as follows.
\begin{proof}[Proof of Theorem \ref{thm:main-1}]
For $s>0$, let $$E(s):=\left\{x\in\mathbb{R}^n~|~\frac12x^TAx<s\right\}.$$ Fix $\bar{s}>0$ such that $\overline{D}\subset E(\bar{s})$. Then recalling \eqref{eq:ordinary-solution}, for $\alpha>0,\beta\in\mathbb{R}$, $$\omega_\alpha(x)=\beta+\int_{\bar{s}}^{\frac12x^TAx}v_\alpha(t)\,dt.$$
By Proposition \ref{pro:omega-alpha}, we have that if $k-l=1$ and $H_k-h_l<\frac12$ or $k-l\geq2$, then $\omega_\alpha$ is a smooth $k$-convex subsolution of \eqref{eq:pro} in $\mathbb{R}^n\setminus\{0\}$ and satisfies $$\omega_\alpha(x)=\frac12x^TAx+\mu_1(\alpha)+O(|x|^{2-2\mathcal{H}_{k,l}}), \quad |x|\to\infty.$$ Also, the function $\mu_1(\alpha)$ is increasing and satisfies
\begin{equation}\label{eq:mu-alpha}
\lim_{\alpha\to\infty}\mu_1(\alpha)=\infty.
\end{equation}

Set
\begin{gather*}
\beta:=\min\{w_{\xi}(x)~|~\xi\in\partial D, x\in \overline{E(\bar{s})}\setminus D\},\\
\hat{b}:=\max\{w_{\xi}(x)~|~\xi\in\partial D, x\in \overline{E(\bar{s})}\setminus D\}.
\end{gather*}
where $w_\xi(x)$ is given by Lemma \ref{lem:w-xi}. Clearly, there holds that
\begin{equation}\label{eq:omega-beta}
\omega_\alpha\leq\beta,\quad\mathrm{in}\ E(\bar{s})\setminus\overline{D}, \forall \alpha>0.
\end{equation}
We will fix the value of $c_*$ in the proof. First we require that $c_*>\hat{b}$. It follows that $$\mu_1(0)=\beta-\bar{s}<\beta\leq \hat{b}<c_*.$$
Thus, in view of \eqref{eq:mu-alpha}, for every $c>c_*$, There exists a unique $\alpha(c)$ such that 
\begin{equation}\label{eq:mu-alpha-c}
\mu_1(\alpha(c))=c.
\end{equation}
Set $$\underline{w}(x)=\max\{w_{\xi}(x)~|~\xi \in \partial D\}.$$
It is clear from Lemma \ref{lem:w-xi} that $\underline{w}$ is a locally Lipschitz function in $\mathbb{R}^n\setminus D$, and $\underline{w}=\varphi$ on $\partial D$. Since $w_\xi$ is a smooth convex solution of \eqref{eq:pro}, $\underline{w}$ is a viscosity subsolution of equation \eqref{eq:pro} in $\mathbb{R}^n\setminus\overline{D}$. We fix a number $\hat{s}>\bar{s}$, and then choose another number $\hat{\alpha}>0$ such that $$\min_{\partial E(\hat{s})}\omega_{\hat{\alpha}}>\max_{\partial E(\hat{s})}\underline{w}.$$ We require that $c_*$ also satisfies $c_*\geq \mu_1(\hat{\alpha})$ and fix now the value of $c_*$. 

For $c\geq c_*$, we have $\alpha(c)=\mu_1^{-1}(c)\geq\mu_1^{-1}(c_*)\geq\hat{\alpha}$, and thereby 
\begin{equation}\label{eq:omega-c-w}
\omega_{\alpha(c)}\geq \omega_{\hat{\alpha}}>\underline{w},\quad \mathrm{on}\ \partial E(\hat{s}).
\end{equation}
By \eqref{eq:omega-beta}, we have 
\begin{equation}\label{eq:omega-beta-under-w}
\omega_{\alpha(c)}\leq \beta\leq \underline{w}, \quad \mathrm{in}\ E(\bar{s})\setminus\overline{D}.
\end{equation}
Now we define for $c>c_*$, 
\begin{equation*}
\underline{u}(x)=\left\{
\begin{array}{ll}
\max\{\omega_{\alpha(c)}(x),\underline{w}(x)\}, & x\in E(\hat{s})\setminus D,\\
\omega_{\alpha(c)}(x),& x\in \mathbb{R}^n\setminus E(\hat{s}).
\end{array}
\right.
\end{equation*}
We know from \eqref{eq:omega-beta-under-w} that $$\underline{u}=\underline{w},\quad\mathrm{in}\ E(\bar{s})\setminus\overline{D},$$ and in particular 
$$\underline{u}=\underline{w}=\varphi,\quad\mathrm{on}\ \partial D.$$
It follows from \eqref{eq:omega-c-w} that $\underline{u}=\omega_{\alpha(c)}$ in a neighborhood of $\partial E(\hat{s})$. Therefore $\underline{u}$ is locally Lipschitz in $\mathbb{R}^n\setminus D$. Since both $\omega_{\alpha(c)}$ and $\underline{w}$ are viscosity subsolutions of \eqref{eq:pro} in $\mathbb{R}^n\setminus\overline{D}$, so is $\underline{u}$.

For $c>c_*$, define $$\bar{u}(x):=\frac12x^TAx+c,$$ which is a smooth convex solution of \eqref{eq:pro}. By  \eqref{eq:omega-beta-under-w}, we have
$$\omega_{\alpha(c)}\leq\beta\leq\hat{b}<c^*<\bar{u}, \quad\mathrm{on}\ \partial{D}.$$
Also, by \eqref{eq:omega-alpha-1} and \eqref{eq:omega-alpha-2} with \eqref{eq:mu-alpha-c}, 
$$\lim_{|x|\to\infty}\left(\omega_{\alpha(c)}-\bar{u}\right)=0.$$
Hence, applying Lemma \ref{lem:compar}, we deduce 
\begin{equation}\label{eq:omega-c-bar-u}
\omega_{\alpha(c)}\leq \bar{u}, \quad\mathrm{on}\ \mathbb{R}^n\setminus D.
\end{equation}
From \eqref{eq:omega-c-w} and the above, one has, for $c>c^*$,
$$w_\xi\leq\bar{u}, \quad\mathrm{on}\ \partial(E(\hat{s})\setminus D),\forall\xi\in\partial D.$$
Using Lemma \ref{lem:compar} again, we obtain
$$w_\xi\leq\bar{u},\quad\mathrm{in}\ E(\hat{s})\setminus\overline{D},\forall\xi\in\partial D.$$
Therefore, $$\underline{w}\leq\bar{u},\quad\mathrm{in}\ E(\hat{s})\setminus\overline{D}.$$
Combining with \eqref{eq:omega-c-bar-u} and the above, we have 
$$\underline{u}\leq\bar{u},\quad\mathrm{in}\ \mathbb{R}^n\setminus D. $$

For any $c>c^*$, let $\mathcal{S}_\alpha$ denote the set of $v\in \mathrm{USC}(\mathbb{R}^n\setminus D)$ which is the viscosity subsolution of \eqref{eq:pro} in $\mathbb{R}^n\setminus\overline{D}$ satisfying 
\begin{equation*}
v=\varphi\quad\mathrm{on}\ \partial D
\end{equation*}
and 
\begin{equation*}
\underline{u}\leq v\leq \bar{u}\quad\mathrm{in} \ \mathbb{R}^n\setminus D.
\end{equation*}
Apparently, $\underline{u}\in\mathcal{S}_\alpha$. Let
$$u(x):=\sup\{v(x)~|~v\in\mathcal{S}_\alpha\}, x\in \mathbb{R}^n\setminus D.$$
 We have
\begin{equation*}
u(x)\geq\underline{u}=\omega_{\alpha(c)}(x)=\frac12x^TAx+c+O(|x|^{2-\frac{k-l}{H_k-h_l}}),\quad\mathrm{as}\quad x\to\infty,
\end{equation*}
especially, if $k-l\geq 2$,  $$u(x)\geq\underline{u}=\frac12x^TAx+c+O(|x|^{\theta(2-n)}), \quad\mathrm{as}\quad x\to\infty$$
where $\theta\in(\frac{k-l-2}{n-2},1],$
and $$u(x)\leq\bar{u}(x)=\frac12x^TAx+c.$$
This indicates \eqref{eq:asym-1} and \eqref{eq:asym-2}.
 
Next, we prove that $u$ satisfies the boundary condition. It is obvious that $$\liminf_{x\to\xi}u(x)\geq\lim_{x\to\xi}\underline u(x)=\varphi(\xi), \quad\forall\xi\in\partial D.$$
So we only need to prove that $$\limsup_{x\to\xi}u(x)\leq\varphi(\xi),\quad\forall\xi\in\partial D.$$
Let $\omega^{+}\in C^2(\overline{E(\bar{s})\setminus D})$ be defined by 
\begin{equation*}
\left\{
\begin{array}{ll}
\Delta \omega^{+}=0,& in \ E(\bar{s})\setminus\overline{D}\\
\omega^{+}=\varphi, & on\ \partial D\\
\omega^{+}=\max_{\partial E(\bar{s})}\bar{u}=\bar{s}+c, & on\ \partial E(\bar{s}).
\end{array}
\right.
\end{equation*}
It is easy to see that a viscosity subsolution $v$ of \eqref{eq:pro} satisfies $\Delta v>0$ in viscosity sense. Therefore, for every $v\in\mathcal{S}_\alpha$, by $v\leq \omega^{+}$ on $\partial(E(\bar{s})\setminus D)$, we have $$v\leq\omega^{+}\quad\mathrm{in}\ E(\bar{s})\setminus\overline{D}.$$ It follows that $$u\leq \omega^{+}\quad\mathrm{in}\ E(\bar{s})\setminus\overline{D},$$ and then $$\limsup_{x\to\xi}u(x)\leq\lim_{x\to\xi}\omega^{+}(x)=\varphi(\xi),\quad\forall\xi\in\partial D.$$
 
Finally, applying the Perron's method Lemma \ref{lem:perron-m}, we find that $u\in C^0(\mathbb{R}^n\setminus D)$ is a viscosity solution of \eqref{eq:pro}. This completes our proof.
\end{proof}

We remark that Theorem \ref{thm:main-1} is also proved by Li and Li \cite{Li-Li-2017} by introducing a combersome notation $m_{k,l}$, which makes their proof looks rather long and difficult to read. Namely,

$$m_{k,l}(\lambda(A)):=\frac{k-l}{\bar{\xi}_k(\lambda(A))-\underline{\xi}_l(\lambda(A))}.$$ Denote $\lambda(A):=(a_1,\cdots,a_n)$. Then $\bar{\xi}_k$ and $\underline{\xi}_k$ respectively are
$$\bar{\xi}_k(a):=\sup_{\mathbb{R}^n\setminus\{0\}} \frac{\sum_{i=1}^n\sigma_{k-1;i}(a)a_i^2x_i^2}{\sigma_k(a)\sum_{i=1}^na_ix_i^2}$$ and $$\underline{\xi}_k(a):=\inf_{\mathbb{R}^n\setminus\{0\}} \frac{\sum_{i=1}^n\sigma_{k-1};i(a)a_i^2x_i^2}{\sigma_k(a)\sum_{i=1}^na_ix_i^2}.$$
In fact, these two quantities $\bar{\xi}_k$ and $\underline{\xi}_l$ are, respectively, equivalent to our $H_k$ and $h_l$ defined in \eqref{eq:H-h}. Indeed, it is clear that $H_k\geq \bar{\xi}_k$. On the other hand,  
$$\frac{\sum_{i=1}^n\sigma_{k-1};i(a)a_i^2x_i^2}{\sigma_k(a)\sum_{i=1}^na_ix_i^2}=\frac{\sigma_{k-1;i_0}(a)a_{i_0}}{\sigma_k(a)}$$
if we choose $x=(0,\cdots,0,1,0,\cdots 0)$ whose the $i$-th component is $1$ and the others all are $0$. Thus, taking $\frac{\sigma_{k-1;i_0}(a)a_{i_0}}{\sigma_k(a)}=H_k(a)$ implies $H_k\leq\bar{\xi}_k$. Hence, $\bar{\xi}_k=H_k$.  Likewise, we can verify $\underline{\xi}_l=h_l$. 


\section{Proof of Theorem \ref{thm:P-L}}\label{sec:4}

In this section, we continue applying the Perron's method to prove Theorem \ref{thm:P-L}. The key is by making the Legendre transform to use a family of smooth convex supersolutions of Poisson equation, rather than that given by Proposition \ref{pro:omega-alpha}, to construct new subsolutions of \eqref{eq:pro-i} possessing suitable asymptotic property at infinity. The following lemma is needed and its proof can also be found in \cite{Dai-2012}.

\begin{Lem}\label{lem:Dai-s}
Let $D$ be a smooth, bounded, strictly convex open set in $\mathbb{R}^n$ and let $\varphi\in C^2(\overline{D})$ be $k$-convex. Assume that $D'\subset\subset D$ is an open subset and $V$ is a locally bounded function in $D$. Then there exists a $k$-convex function $u\in C^2(\overline{D})$ satisfying 
\begin{equation}\label{eq:Dai-S-K-L}
\left\{
\begin{array}{ll}
S_{k,l}(D^2u)\geq 1, & x\in D,\\
u=\varphi, & x\in\partial D,\\
u<V, & x\in D'.
\end{array}
\right.
\end{equation}
\end{Lem}

\begin{proof}
From \cite{Trudinger-1995}, we let $v\in C^2(\overline{D})$ be a $k$-convex solution of problem 
\begin{equation*}
\left\{
\begin{array}{ll}
S_{k,l}(D^2v)=1, & x\in D,\\
v=0, & x\in\partial D.
\end{array}
\right.
\end{equation*}
Since $\Delta v>0$, by the strong maximun principle, one has $v\leq -v_0$ on $\overline{D'}$ for some positive constant $v_0$. Set $$u(x)=\varphi(x)+\alpha v(x),\quad x\in D$$ where $\alpha>0$ is a constant to be determined later.  Then $u\in C^2(\overline{D})$, $u=\varphi$ on $\partial D$ and $$u=\varphi+\alpha v\leq \sup_{D'}\varphi-\alpha v_0<\inf_{D'}V\leq V\quad \mathrm{in}\ D',$$
if $\alpha$ is large enough. Next we claim that $u$ is $k$-convex and is a solution of  $$S_{k,l}(D^2u)\geq 1\quad \mathrm{in} \ D,$$ as long as $\alpha>1$. 
Since $[S_{k,l}(S)]^{\frac{1}{k-l}}$ is a concave function of the elements of the symmetric matrix $S$ whenever $\lambda(S)\in \overline{\Gamma_k}$, we obtain
$$\left[S_{k,l}\left(\frac12(D^2\varphi+\alpha D^2v)\right)\right]^{\frac{1}{k-l}}\geq \frac12\left[S_{k,l}(D^2\varphi)\right]^{\frac{1}{k-l}}+\frac12\left[S_{k,l}(\alpha D^2 v)\right]^{\frac{1}{k-l}},$$
which indicates
\begin{align*}
\left[S_{k,l}(D^2\varphi+\alpha D^2v )\right]^{\frac{1}{k-l}}&\geq [S_{k,l}(D^2\varphi)]^{\frac{1}{k-l}}+[S_{k,l}(\alpha D^2 v)]^{\frac{1}{k-l}}\\
&\geq [S_{k,l}(\alpha D^2 v)]^{\frac{1}{k-l}}.
\end{align*} 
Namely, $S_{k,l}(D^2u)\geq \alpha^{k-l}S_{k,l}(D^2 v)=\alpha^{k-l}$. Thus, if $\alpha>1$, then $S_{k,l}(D^2u)\ge1$ and our claim is true. 

As argued above, we deduce that $u=\varphi+\alpha v$ is   the solution of problem \eqref{eq:Dai-S-K-L} when $\alpha>1$ is sufficiently large.

\end{proof}

\begin{proof}[Proof of Theorem \ref{thm:P-L}]
(i) Suppose $n\geq 3$. We split our proof into three steps.

\textbf{Step\, 1.} Construct a family of smooth convex subsolutions of $$S_{n,n-1}(D^2u)=1$$
in some exterior domain.

Given $A\in\mathcal{A}_{n,n-1}$, it is easy to check that $A^{-1}\in \mathcal{A}_{1,0}$. For $\gamma<0$, let 
\begin{equation}\label{eq:u-y-poisson}
\bar{u}_{\alpha,\gamma,c}(y)=\left\{
\begin{array}{ll}
\frac12y^TA^{-1}y-c+\alpha|y|^{\gamma}, & 2-n\leq\gamma<0,\\\\
\frac12y^TA^{-1}y-c-\alpha|y|^{\gamma}, & \gamma\leq 2-n,
\end{array}
\right.
\end{equation}
where $c\in\mathbb{R}$ and $\alpha>0$.
Obviously, $\Delta\bar{u}_{\alpha,\gamma,c}\leq 1$ in $\mathbb{R}^n\setminus\{0\}$.  Moreover, after a direct computation, for some $0<\delta<\min\{\lambda(A^{-1})\}:=\Lambda\leq \frac{1}{n}$, we have 
$$D^2\bar{u}_{\alpha,\gamma,c}>\delta I,\quad\mbox{ if}~~ |y|>K,$$ where 
\begin{equation*}
K:=\left\{
\begin{array}{ll}
{\big(\frac{-\alpha\gamma}{\Lambda-\delta}\big)}^{\frac{1}{2-\gamma}}, &\mathrm{if} \ 2-n\leq\gamma<0,\\\\
{\big[\frac{\alpha\gamma(\gamma-1)}{\Lambda-\delta}\big]^{\frac{1}{2-\gamma}}}, &\mathrm{if} \ \gamma\leq 2-n. 
\end{array}
\right.
\end{equation*}
We then extend $\bar{u}_{\alpha,\gamma,c}(y)$ smoothly from  $\mathbb{R}^n\setminus B_{K+\epsilon}$ ($\epsilon>0$) to $\mathbb{R}^n$ (still denoted by $\bar{u}_{\alpha,\gamma,c}$) such that $$D^2\bar{u}_{\alpha,\gamma,c}>\delta I\quad 
\mathrm{in}\  \mathbb{R}^n.$$ 

We define the coordinate transformation 
\begin{equation}\label{eq:y-x-D-L}
y\mapsto x=D\bar{u}_{\alpha,\gamma,c}(y)=A^{-1}y+O(|y|^{\gamma-1}).
\end{equation} 
Since the Jacobian $\det D_yx=\det D^2 \bar{u}_{\alpha,\gamma,c}\neq 0$ and 
\begin{align*}
|D(\bar{u}_{\alpha,\gamma,c}(y)-\bar{u}_{\alpha,\gamma,c}(y'))|&=\left|\int_0^1D^2\bar{u}_{\alpha,\gamma,c}(y'+t(y-y'))(y-y')\,dt\right|\\
&\geq\delta|y-y'|,
\end{align*}
 for all $y,y'\in\mathbb{R}^n$. Therefore, the map \eqref{eq:y-x-D-L} is bijective. 
For any $c\in\mathbb{R}$ and $\alpha>0$, we make a Legendre transform
\begin{equation}\label{eq:legendre-t}
u_{\alpha,\gamma,c}(x)=x\cdot y(x)-\bar{u}_{\alpha,\gamma,c}(y(x)).
\end{equation}
 It is known that $y=Du_{\alpha,\gamma,c}(x)$ and $D^2\bar{u}_{\alpha,\gamma,c}=(D^2u_{\alpha,\gamma,c})^{-1}$. Setting $\tilde{D}:=D\bar{u}_{\alpha,\gamma,c}(B_{K+\epsilon})$, we thus obtain that
$$S_{n,n-1}(D^2u_{\alpha,\gamma,c})\geq1,\quad 0<D^2u_{c,\alpha}< \delta^{-1}I, \quad\mathrm{in}\ \mathbb{R}^n\setminus \tilde{D}.$$
Namely, we find that $u_{\alpha,\gamma,c}(x)$ are a family of smooth convex subsolutions of \eqref{eq:pro-i} in $\mathbb{R}^n\setminus \tilde{D}$. Furthermore, by virtue of \eqref{eq:legendre-t} with \eqref{eq:u-y-poisson} and \eqref{eq:y-x-D-L}, we derive that $y=Ax+O(|x|^{\gamma-1})$ and
\begin{equation}\label{eq:u-c-alpha-inf}
u_{\alpha,\gamma,c}(x)=\frac12x^TAx+c+O(|x|^{\gamma}),\quad |x|\to\infty.
\end{equation}

\textbf{Step\, 2.} Construct a viscosity subsolution and a viscosity supersolution to \eqref{eq:pro-i}.
  
For small $\alpha>0$ and $\epsilon>0$, we can assume that $\tilde{D}\subset D$.  Fix now such $\alpha$ and choose a $R>0$ such that $D\subset D\bar{u}_{\alpha,\gamma,c}(B_R):=\hat{D}$. Set
\begin{gather*}
\underline{w}(x)=\max\{w_{\xi}(x)~|~\xi \in \partial D\}.
\end{gather*}
where $w_{\xi}$ is given by Lemma \ref{lem:w-xi}. In view of \eqref{eq:u-y-poisson} and \eqref{eq:legendre-t}, $$\lim_{c\to+\infty} u_{\alpha,\gamma,c}(x)=+\infty,\quad x\in \partial \hat{D}.$$ So we can choose a $c_*$, depending on $\alpha, \gamma, A, D,\varphi$, such that for every $c>c_*$, 
\begin{equation}\label{eq:O-u-c-alpha}
\min_{\partial \hat{D}}u_{\alpha,\gamma,c}(x)>\max_{\partial \hat{D}}\underline{w}.
\end{equation}
Then from Lemma \ref{lem:Dai-s} there is a $n$-convex function $\tilde{u}$ fulfilling 
\begin{equation}\label{eq:Dai-s}
\left\{
\begin{array}{ll}
S_{n,n-1}(D^2\tilde{u})\geq 1, & x\in \hat{D},\\
\tilde{u}=u_{\alpha,\gamma,c}, & x\in\partial \hat{D},\\
\tilde{u}<\underline{w}, & x\in D',
\end{array}
\right.
\end{equation}
where $D\subset D'\subset\subset \hat{D}$. 

Now for every $c>c_*$, define
\begin{equation*}
\underline{u}(x)=\left\{
\begin{array}{ll}
\max\{\underline{w}(x),\tilde{u}\}, & x\in \hat{D}\setminus D,\\
u_{\alpha,\gamma,c}(x),& x\in \mathbb{R}^n\setminus\hat{D}.
\end{array}
\right.
\end{equation*}
By \eqref{eq:Dai-s}, $$\underline{u}=\underline{w},\quad\mathrm{in}\ D'\setminus D,$$ and in particular 
$$\underline{u}=\underline{w}=\varphi,\quad\mathrm{on}\ \partial D.$$
Note from \eqref{eq:O-u-c-alpha} that $\underline{u}=u_{\alpha,\gamma,c}(x)$ in a neighborhood of $\partial \hat{D}$. Therefore $\underline{u}$ is locally Lipschitz in $\mathbb{R}^n\setminus D$. Since both $\tilde{u}$ and $\underline{w}$ are viscosity subsolutions of \eqref{eq:pro-i}, so is $\underline{u}$.

Next, for $c>c_*$, we define $$\bar{u}(x):=\frac12x^TAx+c.$$ 
 Let $\max_{\partial D}\varphi \leq c_*<c\leq\bar{u}$, that is, $\underline{u}\leq \bar{u}$ on $\partial D$. By \eqref{eq:u-c-alpha-inf}, $$\lim_{|x|\to\infty}(\bar{u}-u_{\alpha,\gamma,c})=0.$$
Applying Lemma \ref{lem:compar}, we thus have for $c>c_*$,  $$\underline{u}\leq \bar{u}\quad\mathrm{in}\ \mathbb{R}^n\setminus\overline{D}.$$

\textbf{Step\, 3.} Apply the Perron's method, Lemma \ref{lem:perron-m}, to obtain the existence of viscosity solution of \eqref{eq:pro-i}. This part is the same as the proof of Theorem \ref{thm:main-1}, we thus omit that.

(ii) Suppose $n=2$. Notice that $\det(A-I)=1$ for $A\in\mathcal{A}_{2,1}$, which implies $A-I\in \mathcal{A}_{2,0}$. We consider the Monge-Amp\`ere equation:
\begin{equation}\label{eq:pro-M-w}
\left\{
\begin{array}{ll}
\det(D^2w)=1,& in \quad\mathbb{R}^2\setminus\overline{D},\\
w=\varphi-\frac12|x|^2, & on\quad\partial D.
\end{array}
\right.
\end{equation}
From Theorem 1.1 in \cite{Bao-Li-2012}, there exists some constant $\alpha^*$ depending only on $A-I,b,D$ and $\|\varphi\|_{C^2(\partial D)}$ such that for every $\alpha>\alpha^*$, there exists a unique local convex solution $w\in C^\infty(\mathbb{R}^2\setminus \overline{D})\cap C^0(\mathbb{R}^2\setminus D)$ of \eqref{eq:pro-M-w} that satisfies 
$$O(|x|^{-2})\leq w(x)-V'(x)\leq M(\alpha)+O(|x|^{-2}),\quad |x|\to\infty,$$
where 
$$V'(x)=\frac12x^T(A-I)x+b\cdot x+\alpha\ln\sqrt{x^T(A-I)x}+c(\alpha),$$
and $M(\alpha),c(\alpha)$ are functions of $\alpha$.

Now, let $u(x)=w(x)+\frac12|x|^2$. After a direct computation, we find $u(x)$ is the local convex solution of \eqref{eq:pro-i} when $n=2$. Moreover, 
$$O(|x|^{-2})\leq u(x)-V(x)\leq M(\alpha)+O(|x|^{-2}),\quad |x|\to\infty,$$
where 
$$V(x)=V'(x)+\frac12|x|^2=\frac12x^TAx+b\cdot x+\alpha\ln\sqrt{x^T(A-I)x}+c(\alpha).$$
Thus, our proof is finished.
\end{proof}


\section*{Appendix: Examples for $H_k-h_{k-1}\geq \frac12$}

To study the existence of solutions for exterior Dirichlet problem \eqref{eq:pro} by the Perron's method, the key is to find enough subsolutions of Hessian quotient equations outside a bounded domain of $\mathbb{R}^n$. From the proof of Theorem \ref{thm:main-1}, we can see that when $k-l=1$, the assumption that $H_k-h_l<\frac12$ plays a core role in characterizing the asymptotic behavior at infinity of subsolutions. When $H_k-h_{k-1}\geq \frac12$, even though we have the corresponding subsolution with specific asymptotic behavior, see Proposition \ref{pro:omega-alpha}, we still can not use the current Perron's method to build the existence of Hessian quotient equations \eqref{eq:pro} for $A\in\mathcal{A}_{k,k-1}$. 

Namely, when $k-l=1$ and $H_k-h_l<\frac{k-l}{2}$, since $\frac12x^TAx+c$ is a smooth solution (so is a supersolution) while
\begin{equation*}
\omega_\alpha(x)=\frac12x^TAx+\mu_1(\alpha)+O(|x|^{2-\frac{k-l}{H_k-h_l}}), \quad |x|\to\infty.
\end{equation*}
is a $k$-convex subsolution, see Proposition \ref{pro:omega-alpha}, we can choose a suitable constant $\alpha>0$ such that $\mu_1(\alpha)=c$ and function $\frac12x^TAx+c$ can control $\omega_\alpha(x)$ at infinity, which makes Perron's method available. However, when $H_k-h_{k-1}=\frac12$ or  $H_k-h_{k-1}>\frac12$, although we can construct some subsolutions $\omega_\alpha$ of \eqref{eq:pro}, respectively, satisfying 
\begin{equation}\label{eq:app-omega-alpha-4}
\omega_\alpha(x)=\frac12x^TAx+\alpha\ln \left(\frac12x^TAx\right)+\mu_3(\alpha)+O(|x|^{-2}),
\end{equation}
and
\begin{equation}\label{eq:app-omega-alpha-5}
\omega_\alpha(x)=\frac12x^TAx+\frac{\alpha}{1-\frac{1}{2(H_k-h_l)}}\left(\frac12x^TAx\right)^{1-\frac{1}{2(H_k-h_l)}}+\mu_4(\alpha)+O(|x|^{2\theta}), 
\end{equation}
where $\theta\in(-1,0)$,
apparently, the special solution $\frac12x^TAx+c$ is unable to control subsolution \eqref{eq:app-omega-alpha-4} and \eqref{eq:app-omega-alpha-5} at infinity, whatever $\alpha$ we choose. This is why we do not currently build the existence theorey for problem \eqref{eq:pro} with general diagonal matrix of $\mathcal{A}_{k,k-1}$ and $H_k-h_{k-1}\geq\frac12$. 

For a given $A\in \mathcal{A}_{k,l}$, let $a:=\lambda(A)=(\lambda_1,\lambda_2,\cdots,\lambda_n)$ and $0<\lambda_1\leq\lambda_2\leq\cdots\leq\lambda_n$ denoting its positive eigenvalues. By the definition of \eqref{eq:H-h}, we easily calculate  that, for $1\leq l<k\leq n$,
\begin{equation}\label{eq:A-H-h}
 H_k=\frac{\sigma_{k-1;n}(a)\lambda_n}{\sigma_k(a)},\quad h_l=\frac{\sigma_{l-1;1}(a)\lambda_1}{\sigma_l(a)}.
\end{equation}

\begin{Ex}
When $n=k=2$ and $l=1$, we have, if $A\in \mathcal{A}_{2,1}$,
\begin{gather*}
\lambda_1\lambda_2=\lambda_1+\lambda_2, \quad H_2=1, \quad h_1=\frac{1}{\lambda_2}.
\end{gather*}
This implies $\lambda_2\geq 2$ and thus $H_2-h_1\geq \frac12$. Obviously, $H_2-h_1=\frac12$ if and only if $\lambda_1=\lambda_2=c^*(2,1)=2$, provided $n=2$. 
\end{Ex}

\begin{Ex}
When $n\geq 3$, it is known that $H_k-h_{k-1}=\frac{1}{n}<\frac12$ if $A=c^*(k,k-1)I$. 
\end{Ex}

\begin{Ex}
 Let $A_1=\mathrm{diag}(2,4,4)$. By computations, we have
$$\sigma_3(A_1)=32=\sigma_2(A_1),$$
so $A_{1}\in\mathcal{A}_{3,2}$. Then by \eqref{eq:A-H-h}, we obtain
$$H_3(A_1)=1, h_2(A_1)=\frac12.$$
Thus, $H_3-h_2=\frac12$ for $A_1=\mathrm{diag}(2,4,4)$. 
\end{Ex}

\begin{Ex}
Let $A_2=\mathrm{diag}(\frac{5}{3},5,5)$. We have
$$\sigma_3(A_2)=\frac{125}{3}=\sigma_2(A_2),\  \mbox{and}\ H_3(A_2)=1, h_2(A_2)=\frac{2}{5}. $$
This implies $H_3-h_2=\frac35>\frac12$ for $A_2=\mathrm{diag}(\frac{5}{3},5,5)$.
\end{Ex}

\end{document}